\documentclass{amsart}
\usepackage{graphicx, color, epsfig}
\usepackage{amscd}
\usepackage{amsmath,amsthm,empheq}
\usepackage{amsfonts}
\usepackage{amssymb}
\usepackage{mathrsfs}
\usepackage{graphicx}
\usepackage[all]{xy}
\numberwithin{equation}{section}
\usepackage[latin1]{inputenc}
\usepackage[english]{babel}

\newtheorem{corollary}{Corollary}[section]
\newtheorem{definition}[corollary]{Definition}

\newtheorem{lemma}[corollary]{Lemma}

\newtheorem{thm}[corollary]{Theorem}

\newfont{\sBlackboard}{msbm10 scaled 900}

\newcommand{\mylabel}[1]{\label{#1}
            \ifx\undefined\stillediting
            \else \fbox{$#1$}\fi }
\newcommand{\BE}{\begin{equation}}

\newcommand{\EEQ}{\end{equation}}
\newcommand{\rfb}[1]{\mbox{\rm
   (\ref{#1})}\ifx\undefined\stillediting\else:\fbox{$#1$}\fi}
\newcommand{\half}   {{\frac{1}{2}}}

\newfont{\Blackboard}{msbm10 scaled 1200}

\newfont{\roma}{cmr10 scaled 1200}
\def\RR{{\mathbb R} }

\newcommand{\NN}{\mathbb{N}}

\newcommand{\di}{\displaystyle}

\newcommand{\ee}{\mathcal E}

\newcommand{\ri}{\rightarrow}

\newcommand{\bb}{\begin{equation}}
\newcommand{\bbb}{\end{equation}}

\newcommand{\WW}{{\mathcal W}}

\newcommand{\mm}    {{\hbox{\hskip 0.5pt}}}

\newcommand{\bluff} {{\hbox{\raise 15pt \hbox{\mm}}}}

scaled\magstep2

\makeatletter
\def\section{\@startsection {section}{1}{\z@}{-3.5ex plus -1ex minus
    -.2ex}{2.3ex plus .2ex}{\large\bf}}
\makeatother

\begin{document}
\title[Double phase transonic flow problems]{Double phase transonic flow problems with variable growth: nonlinear patterns and stationary waves}
\author[A. Bahrouni]{Anouar Bahrouni}
\address[A. Bahrouni]{Mathematics Department, University of Monastir,
Faculty of Sciences, 5019 Monastir, Tunisia} \email{\tt
bahrounianouar@yahoo.fr}
\author[V.D. R\u{a}dulescu]{Vicen\c{t}iu D. R\u{a}dulescu}
\address[V.D. R\u{a}dulescu]{Faculty of Applied Mathematics, AGH University of Science and Technology, 30-059 Krak\'ow, Poland \& Institute of Mathematics, Physics and Mechanics, 1000 Ljubljana, Slovenia \& Institute of Mathematics ``Simion Stoilow" of the Romanian Academy, 014700 Bucharest, Romania}
\email{\tt   vicentiu.radulescu@imfm.si}
\author[D.D. Repov\v{s}]{Du\v{s}an D. Repov\v{s}}
\address[D.D. Repov\v{s}]{Faculty of Education and Faculty of Mathematics and Physics, University of Ljubljana, Slovenia \& Institute of Mathematics, Physics and Mechanics, 1000 Ljubljana, Slovenia}
\email{\tt dusan.repovs@guest.arnes.si}
\begin{abstract} In this paper we are concerned with a class of double phase energy functionals arising in the theory of transonic flows. Their main feature is that
the associated Euler equation is driven by the Baouendi-Grushin operator with variable coefficient.
This partial differential equation is of mixed type and possesses both elliptic and hyperbolic regions.
After establishing a weighted inequality for the Baouendi-Grushin operator and a related compactness property, we establish the existence of stationary waves under arbitrary  perturbations  of  the reaction. A description of a related transonic flow model can be found in G.-Q.G. Chen, M. Feldman (2015), {\it Philos. Trans. Roy. Soc. A} {\bf 373}:20140276 (arXiv:1412.1509 [math.AP]).
\end{abstract}
\keywords{Baouendi-Grushin operator, Caffarelli-Kohn-Nirenberg inequality, transonic flow, nonlinear eigenvalue problem, variable exponent.\\
\phantom{aa} 2010 AMS Subject Classification: Primary 35J70,
Secondary 35P30, 76H05}
\maketitle

\section{Unbalanced problems and double phase variational integrals}
The present paper was motivated by recent  fundamental progress in the mathematical analysis of nonlinear models with unbalanced growth. We point out the early works of Marcellini \cite{marce1,marce2} who was interested in qualitative properties, such as lower semicontinuity and regularity of minimizers in the abstract setting of quasiconvex integrals. Related problems were inspired by models arising in nonlinear elasticity and they describe the deformation of an elastic body, see Ball \cite{ball1,ball2}.

We recall some basic facts concerning double phase problems. Let $\Omega\subset \RR^N$ ($N\geq 2$) be a bounded domain with smooth boundary. Let $u:\Omega\to\RR^N$ denote the displacement and assume that $D u$ is the $N\times N$  matrix associated to the deformation gradient. It follows that the total energy is described by an integral of the following type
\bb\label{paolo}I(u)=\int_{\Omega} f(x,D u(x))dx.\bbb
Here, the potential $f=f(x,\xi):\Omega\times\RR^{N\times N}\to\RR$ is assumed to be a quasiconvex function with respect to the second variable; we refer to Morrey \cite{morrey} for details.

Ball \cite{ball1,ball2} was interested in potentials given by
$$f(\xi)=g(\xi)+h({\rm det}\,\xi),$$
where ${\rm det}\,\xi$ denotes the determinant of the $N\times N$ matrix $\xi$. It is also assumed that $g$ and $h$ are nonnegative convex functions satisfying the growth hypotheses
$$g(\xi)\geq c_1\,|\xi|^p\quad\mbox{and}\quad\lim_{t\to+\infty}h(t)=+\infty,$$
where $c_1>0$  and $1<p\leq N$. We point out that the assumption $p\leq N$ was necessary in order to study the existence of cavities for equilibrium solutions, that is, minima of the energy functional \eqref{paolo} which are discontinuous at one point where a cavity appears. In fact, every function $u$ with finite energy belongs to the function space $W^{1,p}(\Omega,\RR^N)$, hence it is  continuous if $p>N$. Accordingly, Marcellini \cite{marce1,marce2} considered  functions $f=f(x,\xi)$ with different growth near the origin and at infinity (unbalanced growth), which satisfy the hypothesis
$$c_1\,|\xi|^p\leq |f(x,\xi)|\leq c_2\,(1+|\xi|^q)\quad\mbox{for all}\ (x,\xi)\in\Omega\times\RR,$$
where $c_1$, $c_2$ are positive constants and $1\leq p\leq q$. Regularity and existence of solutions of elliptic equations with $(p,q)$--growth conditions were studied in \cite{marce2}.

The analysis of non-autonomous energy functionals with energy density changing its ellipticity and growth properties according to the point was developed in several remarkable papers by Mingione {\it et al.} \cite{mingi1}--\cite{mingi5}. These contributions are related to the works of Zhikov \cite{zhikov1}, and they describe the
nature of certain phenomena arising in nonlinear
elasticity.
For instance, Zhikov was interested in providing models for strongly anisotropic materials in the framework of homogenization.
The associated functionals also demonstrated their importance in the study of duality theory
as well as in the context of the Lavrentiev phenomenon \cite{zhikov2}. In relationship with these research directions, Zhikov introduced three different model
functionals, mainly in the context of the Lavrentiev phenomenon. These models are the following:
\bb\label{mingfunc}\begin{array}{ll}
{\mathcal M}(u)&\di :=\int_{\Omega} c(x)|D u|^2dx,\quad 0<1/c(\cdot)\in L^t(\Omega),\ t>1\\
{\mathcal V}(u)&\di :=\int_{\Omega} |D u|^{p(x)}dx,\quad 1<p(x)<\infty\\
{\mathcal P}_{p,q}(u)&\di :=\int_{\Omega} (|D u|^p+a(x)|D u|^q)\,dx,\quad 0\leq a(x)\leq L,\ 1<p<q.
\end{array}
\bbb

The functional ${\mathcal M}$ is characterized by a loss of ellipticity on the subset of $\Omega$ where the potential $c$ vanishes.  This functional has been studied in relationship with nonlinear equations involving
Muckenhoupt weights. The functional ${\mathcal V}$ is still the object of great interest nowadays and several relevant papers have been developed about it.  We  refer to Acerbi and Mingione \cite{acerbi2} in the context of gradient estimates and  contributions to the qualitative analysis of minimizers of nonstandard energy functionals with variable coefficients. We also point out the abstract setting, respectively the variational analysis developed in the monograph  by R\u adulescu and Repov\v s \cite{radrep} (see also the survey papers \cite{radnla, radnom}).
The energy functional defined by ${\mathcal V}$ has been used to build consistent models for strongly
anisotropic materials: in a material made of different components, the exponent $p(x)$
dictates the geometry of a composite that changes its hardening exponent according to
the point.  The functional ${\mathcal P}_{p,q}$ defined in \eqref{mingfunc} appears as un upgraded version of ${\mathcal V}$. Again, in this
case, the modulating potential $a(x)$ controls the geometry of the composite made by
two differential materials, with corresponding hardening exponents $p$ and $q$.

Following Marcellini's terminology, the functionals defined in \eqref{mingfunc} belong to the realm of energy functionals with
nonstandard growth conditions of $(p, q)$--type. These are functionals of the type defined in relation \eqref{paolo}, where the energy density satisfies
$$|\xi|^p\leq f(x,\xi)\leq  |\xi|^q+1,\quad 1\leq p\leq q.$$

An alternative relevant example of a functional having $(p,q)$--growth  is given by
$$u\mapsto \int_{\Omega} |D u|^p\log (1+|D u|)\,dx,\quad \mbox{for}\ p\geq 1,$$
which can be seen as a logarithmic perturbation of the classical $p$-Dirichlet energy. We refer to  Mingione {\it et al.} \cite{mingi1}--\cite{mingi5} for more details. We also point out the recent paper by Cencelj {\it et al.} \cite{cencelj} in the framework of double phase problems with variable growth.

The main feature of our paper is the study of a class of unbalanced double phase problems with variable coefficient. This problem is strictly connected with the analysis of nonlinear patterns and stationary waves
for transonic flow models. We refer to the pioneering  work of Morawetz \cite{mora1,mora2,mora3} on the theory of transonic fluid flow ---referring to partial differential equations that possess both elliptic and hyperbolic regions--- and this remains the most fundamental mathematical work on this subject. The flow is supersonic in the elliptic region, while a shock wave is created at the boundary between the elliptic and hyperbolic regions. In the 1950s, Morawetz used functional--analytic methods to study boundary value problems for such transonic problems.

Throughout this paper we assume that $\Omega\subset\RR^N$ is a bounded domain with smooth boundary. Let $n$, $m$ be nonnegative integers such that $N=n+m$, hence $\RR^N=\RR^n\times\RR^m$. An element $z\in\Omega$ is written as $z=(x,y)$, where $x\in\RR^n$ and $y\in\RR^m$.
The energy studied in this paper is somehow related to the functional ${\mathcal P}_{p,q}(u)$ defined in \eqref{mingfunc} and is of the type
\bb\label{uu1} \int_{\Omega}\frac{\left|{\rm grad}_{x}u\right|^{G(x,y)}+
\left|x\right|^{\gamma}\left|{\rm grad}_{y}u\right|^{G(x,y)}}{G(x,y)}\,dz.\bbb
 In such a case, the variable coefficient $G(x,y)$ describes the geometry of a composite realized by using two materials
with corresponding behaviour described by $|{\rm grad}_{x}u|^{G(x,y)}$ and $|{\rm grad}_{y}u|^{G(x,y)}$. Then in the region $\{z\in\Omega:\ x\not=0\}$ the material
described by  the second integrand is present. In the opposite case, the material described by the first integrand is the only one that creates
the composite.

We also point out that the integral functional \eqref{uu1} is a {\it double phase} energy with variable coefficient due to the presence of the unbalanced terms $|{\rm grad}_{x}u|^{G(x,y)}$ and $|{\rm grad}_{y}u|^{G(x,y)}$ combined with the weight $|x|^\gamma$.

The  differential operator associated to \eqref{uu1} is {\it degenerate} and of {\it mixed type}. This operator is
\bb\label{uu2}{\rm div}_{x}\left(|{\rm grad}_{x} |^{G(x,y)-2}{\rm grad}_{x}u \right)+{\rm div}_{y}\left(|x|^{\gamma}\left|{\rm grad}_{y} \right|^{G(x,y)-2}{\rm grad}_{y}u\right),\bbb
where $G(x,y)$ is a variable coefficient.

\subsection{Degenerate operators of mixed type}
In 1923, Tricomi \cite{tricomi} began the study of second-order partial differential equations  of mixed type by introducing the operator
$$T:=\frac{\partial^2}{\partial x^2}+x\frac{\partial^2}{\partial y^2}.$$
This operator is elliptic on $\{(x,y)\in\RR^2:\ x>0\}$, hyperbolic on $\{(x,y)\in\RR^2:\ x>0\}$, and degenerate on $\{(x,y)\in\RR^2:\ x=0\}$.

An interesting application of this class of elliptic-hyperbolic differential operators can be found in relationship with the theory of planar transonic flow, see \cite{caracal}.
The associated waves are steep wavefronts that propagate in compressible fluids in which convection dominates
diffusion. They are fundamental in nature, especially in high-speed fluid flows. Many such shock reflection/diffraction problems can be formulated as boundary value
problems involving nonlinear partial differential equations of mixed elliptic-hyperbolic
type.

Numerous attempts have been
made recently in order to study the Tricomi operator $T$ as well as some extensions  obtained either by substituting the
degeneracy $x$ with a function $g(x)$ or by replacing the second-order derivative $\partial^2_y$ with the Laplace operator.

In a different context and for different purposes, Baouendi \cite{baouendi} and Grushin \cite{grushin}
considered other classes of degenerate operators, an example is given by
\bb\label{ambr}\frac{\partial^2}{\partial x^2}+x^{2r}\frac{\partial^2}{\partial y^2},\quad r\in\NN.\bbb
The Baouendi--Grushin operator can be viewed as the Tricomi operator for transonic flow restricted to subsonic regions.
On the other hand, a second-order differential operator $T$ in divergence form on the
plane, can be written as an operator whose principal part is a Baouendi-Grushin-type operator,
provided that the principal part of $T$ is nonnegative and its quadratic form does not
vanish at any point, see \cite{franchi}.

Let us introduce $G_{2r}:=\Delta_x+|x|^{2r}\Delta_y$, with $x\in\RR^n$, $y\in\RR^m$, and $n+m=N$. This operator can be seen as the $N$-dimensional analogue of \eqref{ambr}. If $z=(x,y)\in\RR^N$, we notice that the operator $G_{2r}$ can be rewritten, with a suitable choice of function $a_\alpha$, in the following form:
$${\mathcal L}u:=\sum_{|\alpha|}D^\alpha_z(a_\alpha (z,u)).$$ This class of differential operators was studied by Mitidieri and Pohozaev \cite{miti1} and D'Ambro\-sio \cite{ajfa}.

 Let $\Omega \subset
\mathbb{R}^{N}$ be a bounded domain with smooth boundary. Assume that $N=n+m$. We now introduce the Baouendi--Grushin operator with variable coefficient. Let us consider the continuous function $ G: \overline{\Omega}\rightarrow (1,\infty)$. We introduce the nonhomogeneous  operator $\Delta_{G(x,y)}$ defined by
\begin{align*}
\Delta_{G(x,y)}u&={\rm div}\,({\rm grad}_{G(x,y)}u )\\
&={\rm div}_{x}(\left|{\rm grad}_{x} \right|^{G(x,y)-2}{\rm grad}_{x}u )+{\rm div}_{y}(\left|x\right|^{\gamma}\left|{\rm grad}_{y} \right|^{G(x,y)-2}{\rm grad}_{y} u)\\
&=\displaystyle
\sum_{i=1}^{n}\left(\left|{\rm grad}_{x}
\right|^{G(x,y)-2}u_{x_i}\right)_{x_i}+\left|x\right|^{\gamma} \displaystyle
\sum_{i=1}^{m}\left(\left|{\rm grad}_{y}
\right|^{G(x,y)-2}u_{y_i}\right)_{y_i},
\end{align*}
where
\[
{\rm grad}_{G(x,y)}u={\mathcal A}(x)\begin{bmatrix}
\left|{\rm grad}_{x} \right|^{G(x,y)-2}{\rm grad}_{x}u \\
\left|x\right|^{\gamma}\left|{\rm grad}_{y}
\right|^{G(x,y)-2}{\rm grad}_{y}u
\end{bmatrix}
\]
and
\[
{\mathcal A}(x)=\begin{bmatrix}
I_{n} & O_{n,m}\\
O_{m,n}& \left|x\right|^{\gamma} I_{m}
\end{bmatrix}
\in {\mathcal M}_{N\times N}(\mathbb{R}).
\]

Then the operator $\Delta_{G(x,y)}$ is degenerate  along the $m$-dimensional subspace $M:=\{0\}\times\RR^m$ of $\RR^N$.

The present paper complements our previous contributions related to double phase anisotropic variational integrals, see \cite{dou2,dou1}. This paper also extends our recent results established in \cite{annon} to a mixed elliptic-hyperbolic setting. In this way, Euclidean results mentioned above continue to be a source of inspiration
for the problem of finding analogues and new inequalities on the sub-Riemannian space $\RR^N=\RR^m\times\RR^n$ defined by the sub-elliptic gradient, which is the $N$-dimensional vector field given by
$${\rm grad}_\gamma =({\rm grad}_x,|x|^\gamma{\rm grad}_y)=(X_1,\ldots ,X_m,Y_1,\ldots ,Y_n)$$
with the corresponding Baouendi-Grushin vector fields
$$X_i=\left| {\rm grad}_x\right|^{G(x,y)-2}\frac{\partial}{\partial x_i},\ \ i=1,\ldots ,m$$ and $$Y_j=|x|^\gamma\,\left| {\rm grad}_y\right|^{G(x,y)-2}\frac{\partial}{\partial y_j},\ \ j=1,\ldots ,n.$$

In the isotropic case corresponding to $G(x,y)\equiv 2$, the above vector fields $X_i$ and $Y_j$ are homogeneous of degree one with respect to the dilation
$$X_i(\delta_\lambda)=\lambda\delta_\lambda(X_i),\quad Y_j(\delta_\lambda)=\lambda\delta_\lambda(Y_j),$$
where the anisotropic dilation $\delta_\lambda$ is defined by $\delta_\lambda (x,y):=(\lambda x,\lambda^{1+\gamma}y)$. However, $\Delta_{2}$ is not translation invariant in $\RR^N$ but it is invariant with respect to the translations
along $M$. Assuming that  $G(x,y)\equiv 2$ and $\gamma=1/2$, then the operator $\Delta_{G(x,y)}$ is intimately connected to the sub-Laplacians in groups
of the Heisenberg type. A description of a related transonic flow model can be found in \cite{chen} (arXiv:1412.1509). Finally, we point out that if $\gamma$ is an even positive integer, then the Baouendi-Grushin operator is a sum
of squares of $C^\infty$ vector fields satisfying the H\"ormander condition.

\section{A weighted inequality for  $\Delta_G$}
One of the many significant contributions by Hardy and Littlewood on the subject of inequalities, and in particular integral inequalities involving derivatives of functions, can be found in their joint paper \cite{hardy}. This paper subsequently formed the basis of Chapter VII of the book of Hardy, Littlewood and P\'olya \cite{hardybook}; that chapter is essentially concerned with the applications of the calculus of variations to integral inequalities, but it also involves direct analytical methods required to avoid difficulties of singular problems and unattained bounds in the calculus of variations. In \cite{caffa} Caffarelli, Kohn and Nirenberg proved a rather general interpolation inequality with weights, which extends the Hardy-Littlewood inequality.
We refer to \cite{laptev, ambrpams, laptev1} for related inequalities in the context of Grushin-type operators.
In this section, motivated by the results obtained in our recent paper \cite{annon} in the  framework  of  variable  exponent, we first establish  a Caffarelli-Kohn-Nirenberg inequality
for  $\Delta_G$. Next, we deduce a compactness property of an anisotropic function space. This abstract result will play a key role in the mathematical analysis of a boundary value problem driven by the Baouendi-Grushin operator.

We define $$G^+:=\sup\{G(x,y):\ (x,y)\in\Omega\}\quad \mbox{and}\quad G^-:=\inf\{G(x,y):\ (x,y)\in\Omega\}.$$
We suppose that the domain $\Omega$ intersects the degeneracy set $[x=0]$, that is, $$\Omega\cap\{(0,y):\ y\in\RR^m\}\not=\emptyset.$$ Throughout this paper, we denote by $|\,\cdot\,|_{p(z)}$ the norm in the
Lebesgue space with variable exponent $L^{p(z)}(\Omega)$. For general properties of function spaces with variable exponent we refer to \cite{radrep}.

\begin{thm}\label{cafg}
Assume that $G$ is a function of class $C^{1}$ and that $G(x,y)\in (2,N)$ for all
$(x,y)\in\Omega$. Then there exists a positive
constant $\beta$ such that for all $u\in C^{1}_{c}(\Omega)$
$$\begin{array}{ll}
\displaystyle \int_{\Omega}(1+|x|^{\gamma})|u|^{G(x,y)}\,dxdy &\leq
\beta \displaystyle \int_{\Omega}(|{\rm grad}_{x}
u|^{G(x,y)}+ |x|^{\gamma}|{\rm grad}_{y} u|^{G(x,y)})dxdy \\
&+\beta \displaystyle
\int_{\Omega}|u|^{G(x,y)-1}(1+u^2)(|{\rm grad}_{x}
G(x,y)|+|x|^{\gamma}|{\rm grad}_{y}
G(x,y)|)\,dxdy.
\end{array}
$$
\end{thm}

\begin{proof}
We define the functions $W_{1},W_{2}:\mathbb{R}^{n}\times
\mathbb{R}^{m}\rightarrow \mathbb{R}^{n}\times \mathbb{R}^{m}$ by
$$W_{1}(x,y)=(x,0_{m})=:(I_{1}(x),0_{m})\ \ \mbox{and} \ \ W_{2}(x,y)=(0_{n},y)=:(0_{n},I_{2}(y)).$$ Then
for all $(x,y)\in\Omega$
\bb\label{comp}\begin{array}{ll}
{\rm div}\, (W_{1}(x,y)\left|u\right|^{G(x,y)})&= |u|^{G(x,y)}
{\rm div}\,(W_{1})
+\displaystyle G(x,y)|u|^{G(x,y)-2} u {\rm grad}
u\cdot W_{1}\\
&\ \\
&+|u|^{G(x,y)} \log(|u|) {\rm grad} G\cdot W_{1}
\\
&\ \\
&= |u|^{G(x,y)}
{\rm div_{x}}\,(I_{1}(x))
+\displaystyle G(x,y)|u|^{G(x,y)-2} u {\rm grad}_{x}
u\cdot I_{1}(x)\\
&\ \\
&+|u|^{G(x,y)} \log(|u|) {\rm grad}_{x} G(x,y)\cdot I_{1}(x)
\end{array}
\bbb
and
\bb\label{comp18}\begin{array}{ll}
{\rm div}\,(|x|^{\gamma}|u|^{G(x,y)} W_{2}(x,y)) &=|x|^{\gamma}
|u|^{G(x,y)}
{\rm div_{y}}\,(I_{2}(y))\\
 &\ \\
&+\displaystyle |x|^{\gamma} G(x,y)|u|^{G(x,y)-2} u {\rm grad}_{y}
u\cdot I_{2}(y)\\
&\ \\
&+|x|^{\gamma}|u|^{G(x,y)} \log(|u|) {\rm grad}_{y} G(x,y)\cdot I_{2}(y).
\end{array}
\bbb
By the flux-divergence theorem we have for all $u\in C^{1}_{c}(\Omega)$
\bb\label{comp19}\displaystyle \int_{\Omega}{\rm div}\,(|u|^{G(x,y)} W_{1}(x,y))\,dxdy=\int_{\Omega}{\rm div}\,(|x|^{\gamma}|u|^{G(x,y)}
W_{2}(x,y))\,dxdy=0.\bbb
Combining relations \eqref{comp}--\eqref{comp19} and  \cite[Lemma 3.1]{annon}, we can deduce that
\begin{align}
\displaystyle  \int_{\Omega} |u|^{G(x,y)} &{\rm
{\rm div}_{x}}\,(I_{1}(x))\,dxdy \leq \displaystyle
\int_{\Omega}|u|^{G(x,y)}|\log(|u|)||{\rm grad}_{x}
G(x,y)||I_{1}(x)|\,dxdy \nonumber\\
&+G^{+}\displaystyle \int_{\Omega}|u|^{G(x,y)-1}|{\rm grad}_{x}
u| |I_{1}(x)|\,dxdy \nonumber \\
&\leq \mu \|W_{1}\|_{L^{\infty}(\Omega)} \displaystyle
\int_{\Omega}\left| {\rm grad}_{x} G(x,y)\right| |u|^{G(x,y)-1}(u^2+1)\,dxdy
\nonumber\\
&+  \epsilon G^{+}\|W_{1}\|_{L^{\infty}(\Omega)}
\displaystyle \int_{\Omega}|u|^{G(x,y)}\,dxdy +
G^{+}\frac{\|W_{1}\|_{L^{\infty}(\Omega)}}{\epsilon^{G^{-}-1}}
\displaystyle \int_{\Omega}|{\rm grad}_{x} u|^{G(x,y)}\,dx dy.\nonumber
\end{align}
Similarly, we have
\begin{align}
\displaystyle  \int_{\Omega}|x|^{\gamma} |u|^{G(x,y)} &
{\rm div}_{y}\,(I_{2}(x))\,dxdy \leq \mu
\|W_{2}\|_{L^{\infty}(\Omega)} \displaystyle
\int_{\Omega}|x|^{\gamma}|{\rm grad}_{y} G(x,y)||u|^{G(x,y)-1}(u^2+1)\,dxdy
\nonumber\\
&+  \epsilon G^{+}\|W_{2}\|_{L^{\infty}(\Omega)}
\displaystyle \int_{\Omega}|x|^{\gamma}|u|^{G(x,y)}dxdy\\
&+G^{+}\frac{\|W_{2}\|_{L^{\infty}(\Omega)}}{\epsilon^{G^{-}-1}}
\displaystyle \int_{\Omega}|x|^{\gamma}|{\rm grad}_{y} u|^{G(x,y)}dx dy.\nonumber
\end{align}

Combining these relations and taking into account
that ${\rm div}\,(W_{1})=n$ and ${\rm div}\,(W_{2})=m$,
 we deduce that
\begin{align*}
[2&-\epsilon
G^{+}(\|W_{1}\|_{L^{\infty}(\Omega)}+\|W_{2}\|_{L^{\infty}(\Omega)})
]\displaystyle \int_{\Omega}(1+|x|^{\gamma})|u|^{G(x,y)}dxdy \leq \nonumber\\
&\displaystyle
G^{+}\frac{\|W_{1}\|_{L^{\infty}(\Omega)}+\|W_{2}\|_{L^{\infty}(\Omega)}}{\epsilon
^{G^{-}-1}}
\int_{\Omega}(|{\rm grad}_{x} u|^{G(x,y)}+ |x|^{\gamma}|{\rm grad}_{y} u|^{G(x,y)})\,dxdy )\nonumber\\
&+ c\displaystyle \int_{\Omega}|u|^{G(x,y)+1}(|{\rm grad}_{x}
G(x,y)|+|x|^{\gamma}|{\rm grad}_{y}
G(x,y)|)\,dxdy\nonumber\\
&+\mu(\|W_{1}\|_{L^{\infty}(\Omega)}+\|W_{2}\|_{L^{\infty}(\Omega)})\displaystyle
\int_{\Omega}|u|^{G(x,y)-1}(|{\rm grad}_{x}
G(x,y)|+|x|^{\gamma}|{\rm grad}_{y} G(x,y)|)\,dxdy,\nonumber
\end{align*}
with $c=\mu (\|W_{1}\|_{L^{\infty}(\Omega)}+
\|W_{2}\|_{L^{\infty}(\Omega)})$. So, by choosing $$\epsilon
<
\frac{2}{G^{+}(\|W_{1}\|_{L^{\infty}(\Omega)}+\|W_{2}\|_{L^{\infty}(\Omega)})},$$
we set
$$\beta=\left(\|W_{1}\|_{L^{\infty}(\Omega)}+\|W_{2}\|_{L^{\infty}(\Omega)}\right)\frac{\max(
\mu,\frac{G^{+}}{\epsilon
^{G^{-}-1}})}
{2-\epsilon G^{+}(\|W_{1}\|_{L^{\infty}(\Omega)}+\|W_{2}\|_{L^{\infty}(\Omega)})}\,.$$
This completes the proof of Theorem \ref{cafg}.
\end{proof}

We denote by ${\mathcal W}$ the closure of
$C_{c}^{1}(\Omega)$ with respect to the norm
\begin{align*}
\left\|u\right\|&=\left|{\rm grad} _{x}u\right|_{G(x,y)}+\left| \left|x\right|^{\frac{\gamma}
{G(x,y)}} {\rm grad}
_{y}u\right|_{G(x,y)}\\&+\left|u(|{\rm grad}_{x}
G(x,y)|+|x|^{\gamma}|{\rm grad}_{y}
G(x,y)|)^{\frac{1}{G(x,y)+1}}\right|_{G(x,y)+1}\\
&+\left|u(|{\rm grad}_{x} G(x,y)|+|x|^{\gamma}|{\rm grad}_{y}
G(x,y)|)^{\frac{1}{G(x,y)-1}}\right|_{G(x,y)-1}.
\end{align*}

We now establish the following compactness property.
\begin{lemma}\label{comgr}
Assume that $G$ is a function of class $C^{1}$ and that $G(x,y)\in (2,N)$ for all
$(x,y)\in\Omega$. Furthermore, suppose that  $0<\gamma<\frac{N(G^{-}-s)}{s}$ and $1<s<G^{-}$. Then the function space ${\mathcal W}$ is
compactly embedded in
$L^{s}(\Omega)$.
\end{lemma}

\begin{proof}
Let $(u_{n})$ be an arbitrary bounded sequence in
${\mathcal W}$. Since the domain $\Omega$ is assumed to  intersect the
degeneracy set $[x=0]$, we deduce that there are $y_{0}\in \mathbb{R}^{m}$
and $R>0$ such that the ball of radius $R$ centered at $(0_{n},y_{0})$ is included in $\Omega.$ Thus, there
exists $0<\epsilon_{0}<\min(1,R)$ such that
$\overline{D_{\epsilon_{0}}}\subset
\overline{B}_{R}(0_{n},y_{0})\subset \overline{\Omega},$ with
$$D_{\epsilon_{0}}=\left\{(x,y)\in B_{R}(0_{n},y_{0}):\
\left|x\right|<\epsilon_{0}\right\}.$$ We fix arbitrarily $\epsilon>0$ with
$\epsilon <\epsilon_{0}$ and  set
$$D_{\epsilon}=\left\{(x,y)\in \mathbb{R}^{n+m}:\ \left|x\right|<\epsilon\ \mbox{and}\  \left|(x,y)-(0_{n},y_{0})\right|<R\right\}.$$
 By Theorem \ref{cafg}, the sequence $(u_{n})$ is bounded in $L^{G(x,y)}(\Omega
\setminus \overline{D_{\epsilon}})$. Consequently, $(u_{n})$
is  bounded in the space $W^{1,G(x,y)}(\Omega\setminus
\overline{D_{\epsilon}})$. Since $W^{1,G(x,y)}(\Omega\setminus
\overline{D_{\epsilon}} )$ is continuously embedded into $ W^{1,G^{-}}(\Omega\setminus
\overline{D}_{\epsilon})$, we  deduce that the sequence $(u_{n})$ is bounded
 in $W^{1,G^{-}}(\Omega\setminus \overline{D_{\epsilon}})$.

 By the Rellich-Kondratchov  embedding theorem, we know that there is a convergent subsequence
 of $(u_{n})$ in $L^{s}(\Omega \setminus \overline{D_{\epsilon}} )$.
  Thus, for any large enough $n$ and $m$ we have
\begin{equation}\label{31}
\displaystyle \int_{\Omega\setminus
\overline{D}_{\epsilon}}|u_{n}-u_{m}|^{s}dx <\epsilon.
\end{equation}
By the H\"older inequality for variable exponents (see \cite[p. 8]{radrep}) we have
\begin{align}
\displaystyle \int_{D_{\epsilon}}\left|u_{m}-u_{n}\right|^{s}dxdy&=
\displaystyle
\int_{D_{\epsilon}}\frac{1}{\left|x\right|^{\frac{s\gamma}
{G(x,y)}}}\left|x\right|^{\frac{s\gamma}{G(x,y)}}\left|u_{m}-u_{n}\right|^{s}dxdy\nonumber\\
&\leq
2\left|\frac{1}{\left|x\right|^{\frac{s\gamma}{G(x,y)}}}\right|_{(\frac{G(x,y)}{s})'}
\left|\left|x\right|^{\frac{s\gamma}{G(x,y)}}\left|u_{m}-u_{n}\right|^{s}\right|_{\frac{G(x,y)}{s}},\nonumber
\end{align}
where $(\frac{G(x,y)}{s})'=\frac{G(x,y)}{G(x,y)-s}.$

By Theorem \ref{cafg} and since for all $(x,y)\in D_{\epsilon}$ we have $\left|x\right|< \epsilon \leq 1$,  we obtain
\begin{align*}
\left|\left|x\right|^{\frac{s\gamma}{G(x,y)}}\left|u_{m}-u_{n}\right|^{s}\right|_{\frac{G(x,y)}{s}}&\leq
\displaystyle
\left(\int_{\Omega}\left|x\right|^{\gamma}\left|u_{m}-u_{n}\right|^{G(x,y)}dxdy\right)^{s/G^{-}}\\&
+\left(\int_{\Omega}\left|x\right|^{\gamma}\left|u_{m}-u_{n}\right|^{G(x,y)}dxdy\right)^{s/G^{+}}
<\infty.
\end{align*}

If $\rho$ denotes the modular function in the Lebesgue space with the variable exponent $G(x,y)/[G(x,y)-s]$ (see \cite[p. 9]{radrep}), we observe that
$$
\left|\frac{1}{\left|x\right|^{\frac{s\gamma}{G(x,y)}}}\right|_{(\frac{G(x,y)}{s})'}\leq
\left[\rho\left(\frac{1}{\left|x\right|^{\frac{s\gamma}{G(x,y)}}
}\right)\right]^{((\frac{G(x,y)}{s})')^{+}}+
\left[\rho\left(\frac{1}{\left|x\right|^{\frac{s\gamma}{G(x,y)}}
}\right)\right]^{((\frac{G(x,y)}{s})')^{-}}.
$$

 Define
$\Omega_{1}=\{x\in \mathbb{R}^{n}, \ \ |x|<\epsilon\}$ and
$\Omega_{2}=\{y\in \mathbb{R}^{m}, \ \ |y-y_{0}|<R\}$. It follows that
$D_{\epsilon}\subset \Omega_{1}\times \Omega_{2}=\{(x,y)\in \mathbb{R}^{n+m}:\ |x|<\epsilon, |y-y_{0}|<R \}.$
Then \begin{align*} \displaystyle
\int_{D_{\epsilon}}\frac{1}{|x|^{\frac{s\gamma}{G(x,y)-s}}}\,dxdy&\leq
\displaystyle \int_{\Omega}|x|^{\frac{-s\gamma}{G^{-}-s}}dxdy \leq
\displaystyle \int_{\Omega_{1}\times
\Omega_{2}}|x|^{\frac{-s\gamma}{G^{-}-s}}dxdy\\
&=|\Omega_{2}|\displaystyle
\int_{\Omega_{1}}|x|^{\frac{-s\gamma}{G^{-}-s}}dx= \displaystyle
\int_{0}^{\epsilon}w_{n}t^{N-1-\frac{s\gamma}{G^{-}-s}}dt=c\epsilon^{N-\frac{s\gamma}{G^{-}-s}},
\end{align*}
where $w_{N}$ is the area of the unit sphere in $\mathbb{R}^{N}.$
Invoking the above estimates, we infer that
$$\left|\frac{1}{\left|x\right|^{\frac{s\gamma}{G(x,y)}}}\right|_{(\frac{G(x,y)}{s})'}\leq
c(\epsilon^{\alpha_{1}}+\epsilon^{\alpha_{1}}),$$ where $\alpha_{1},
\alpha_{2}$ are positive constants. It follows that
$$\displaystyle
\int_{\Omega}|u_{m}-u_{n}|^{s}dxdy\leq
c(\epsilon+\epsilon^{\alpha_{1}}+\epsilon^{\alpha_{1}}).$$ This shows that
$(u_{n})$ is a Cauchy sequence in $L^{s}(\Omega)$, hence the proof is concluded.
\end{proof}

\section{A nonlinear problem driven by $\Delta_G$}

We study the following boundary value problem
\begin{equation}\label{imain}
  \left\{\begin{array}{lll} &\displaystyle -\Delta_{G(x,y)}u+A(x,y)(|u|^{G(x,y)-1}+|u|^{G(x,y)-3})u=\lambda \left|u\right|^{s-2}u
 &\quad\mbox{in}\
\partial\Omega\\
&u=0 &\quad \mbox{on}\
\partial\Omega\,,
\end{array}\right.
\end{equation}
where $\lambda>0$ and $$A(x,y)=|{\rm grad}_{x}
G(x,y)|+|x|^{\gamma}|{\rm grad}_{y}
G(x,y)|\quad \mbox{for all $(x,y)\in \Omega$}.$$

\begin{definition}\label{defi1}
We say that $u\in {\mathcal W} $ is a weak solution of problem
\eqref{imain} if for all $v\in {\mathcal W}\setminus\{0\}$
\begin{align*}
\displaystyle
&\displaystyle\int_{\Omega}\left[\left|{\rm grad}_{x}u\right|^{G(x,y)-2}{\rm grad}_{x}u{\rm grad}_{x}v+\left|x\right|^{\gamma}\left|{\rm grad}_{y}u\right|^{G(x,y)-2}
{\rm grad}_{y}u {\rm grad}_{y}v\right]dxdy+\\
&\displaystyle \int_{\Omega} A(x,y)\left|u\right|^{G(x,y)-3}(u^2+1)uv\,dxdy
=\lambda \displaystyle \int_{\Omega} \left|u\right|^{s-2}uv\,dxdy.
\end{align*}
\end{definition}

We will say that the corresponding real number $\lambda$  for which problem \eqref{imain} has a nontrivial solution is an {\it eigenvalue} and the corresponding $u\in{\mathcal W}\setminus\{0\}$ is an {\it eigenfunction} of the problem. These terms are in accordance with the related notions introduced by Fu\v{c}ik, Ne\v{c}as, Sou\v{c}ek, and Sou\v{c}ek \cite[p. 117]{fucik} in the context of {\it nonlinear} operators. Indeed, if we denote $$\begin{array}{ll}S(u)&\displaystyle :=\int_{\Omega}\frac{1}{G(x,y)}\left[\left|{\rm grad}_{x}u\right|^{G(x,y)}+
\left|x\right|^{\gamma}\left|{\rm grad}_{y}u\right|^{G(x,y)}\right]dxdy\\
&\displaystyle+\int_{\Omega}A(x,y)\left[\frac{\left|u\right|^{G(x,y)+1}}{G(x,y)+1}+
\frac{\left|u\right|^{G(x,y)-1}}{G(x,y)-1}\right]dxdy\end{array}
$$
and
$$T(u):=\int_{\Omega}\left|u\right|^{s-2}uv\,dxdy$$
then $\lambda$ is an eigenvalue for the pair $(S,T)$ of nonlinear operators (in the sense of \cite{fucik}) if and only if
there is a corresponding eigenfunction that is a solution of problem \eqref{imain} as described in Definition \ref{defi1}.

The next result establishes the existence of an infinite interval of eigenvalues.
This property corresponds to {\it arbitrary perturbations} of the reaction term, namely existence of nontrivial solutions with respect to any positive parameter $\lambda$.

\begin{thm}\label{thm}
Assume that $G$ is a function of class $C^{1}$ and that $G(x,y)\in (2,N)$ for all
$(x,y)\in\Omega$. We also suppose that $0<\gamma<\frac{N(G^{-}-s)}{s}$ and $1<s<G^{-}-1$.  Then
 any $\lambda>0$ is an eigenvalue of problem \eqref{imain}.
\end{thm}

The energy functional associated to problem \eqref{imain} is
 $\ee:{\mathcal W} \rightarrow \mathbb{R}$  defined
by
\begin{align*}
\ee(u)&=\displaystyle \int_{\Omega}\frac{1}{G(x,y)}\left[\left|{\rm grad}_{x}u\right|^{G(x,y)}+
\left|x\right|^{\gamma}\left|{\rm grad}_{y}u\right|^{G(x,y)}\right]dxdy\\
&+\displaystyle \int_{\Omega}A(x,y)\left[\frac{\left|u\right|^{G(x,y)+1}}{G(x,y)+1}+\frac{\left|u\right|^{G(x,y)-1}}{G(x,y)-1}\right]dxdy
-\frac{\lambda}{s}\displaystyle \int_{\Omega} \left|u\right|^{s}dxdy.
\end{align*} Then $\ee$ is of class $C^{1}$ and for all $u,v\in {\mathcal W}$
\begin{align*}
\langle \ee'(u),v\rangle &=\displaystyle
\int_{\Omega}\left[\left|{\rm grad}_{x}u\right|^{G(x,y)-2}{\rm grad}_{x}u{\rm grad}_{x}v+\left|x\right|^{\gamma}\left|{\rm grad}_{y}u\right|^{G(x,y)-2}
{\rm grad}_{y}u {\rm grad}_{y}v\right]dxdy\\
&+\displaystyle \int_{\Omega} A(x,y)\left|u\right|^{G(x,y)-3}(u^2+1)uvdxdy
-\lambda \displaystyle \int_{\Omega}\left|u\right|^{s-2}uv\,dxdy .
\end{align*}

We recall that ${\mathcal W}$ is the closure of
$C_{c}^{1}(\Omega)$ under the norm
$$\begin{array}{ll}
\displaystyle\left\|u\right\|=\left|{\rm grad} _{x}u\right|_{G(x,y)}+\left| \left|x\right|^{\frac{\gamma}
{G(x,y)}} {\rm grad}
_{y}u\right|_{G(x,y)}&\displaystyle +\left|uA(x,y)^{1/(G(x,y)+1)}\right|_{G(x,y)+1}\\
&+\displaystyle\left|uA(x,y)^{1/(G(x,y)-1)}\right|_{G(x,y)-1}\,.\end{array}
$$

Thus, taking into account the expression of $\ee':\WW\ri\WW^*$, we can deduce that $\ee'$ is well-defined and bounded.

The proof of Theorem \ref{thm} is based on the following ideas:\\ (i) energy estimates, namely the existence of a ``valley" far from the origin and of a ``mountain" for $\ee$ near the origin;\\ (ii) existence of a negative relative minimum for $\ee$ and a sequence of ``almost critical points" for the energy, for any $\lambda>0$.

The main ingredients of the proof are the compactness property established in Lemma~\ref{comgr} and the Ekeland variational principle, which is the nonlinear version of the Bishop-Phelps theorem.

\begin{lemma}\label{lema} (i) There exists $\phi\geq 0$ in $\WW$ such that $\ee(t\phi)<0$ for all  small enough
$t>0$.

(ii) For all $\lambda>0$, there exist positive numbers $\rho$ and $\alpha$ such that $\ee(u)\geq\alpha$ for all $u\in\WW$ with $\|u\|=\rho$.
\end{lemma}

\begin{proof}
(i) Fix $\phi\in\WW\setminus\{0\}$ with $\phi\geq 0$ and $\|\phi\|<1$. For all $t\in (0,1)$ we have
$$\begin{array}{ll}
\ee(t\phi)&\di\leq\frac{t^{G^-}}{G^-}\int_\Omega |{\rm grad}_x\phi|^{G(x,y)}dxdy+
\frac{t^{G^-}}{G^-}\int_\Omega |x|^\gamma|{\rm grad}_y\phi|^{G(x,y)}dxdy\\
&\di +\frac{t^{G^-+1}}{G^-+1}\int_\Omega A(x,y)\phi^{G(x,y)+1}dxdy+
\frac{t^{G^--1}}{G^--1}\int_\Omega A(x,y)\phi^{G(x,y)-1}dxdy-\lambda\,\frac{t^s}{s}\int_\Omega\phi^sds\\
&\di=C_1t^{G^-}+C_2t^{G^-+1}+C_3t^{G^--1}-\lambda C_4t^s,\end{array}$$
where $C_1$, $C_2$, $C_3$ and $C_4$ are positive numbers.

Since $s<G^--1$, our assertion follows for small enough $t>0$.

(ii) For all $u\in\WW$ we have
\bb\label{munte}\begin{array}{ll}
\ee(u)&\di\geq\frac{1}{G^+}\int_\Omega\left[|{\rm grad}_xu|^{G(x,y)}+|x|^\gamma |{\rm grad}_yu|^{G(x,y)} \right]dxdy\\
&\di+\frac{1}{G^++1}\int_\Omega A(x,y)|u|^{G(x,y)+1}dxdy+\frac{1}{G^+-1}\int_\Omega A(x,y)|u|^{G(x,y)-1}dxdy\\ &\di-\frac{\lambda}{s}\int_\Omega |u|^sds.\end{array}\bbb

By Lemma \ref{comgr}, there exists $\beta >0$ such that
$$\left|u\right|_{s}\leq \beta\left\|u\right\|, \ \ \mbox{for all}\ u\in \WW.$$

We fix $\rho>1$ and  assume that $\|u\|=\rho$. By relation \eqref{munte} we obtain, for a suitable positive constant $C$ depending only on $G^+$ and $G^-$,
$$\ee(u)\geq C\,\|u\|^{G^--1}-\lambda\,\frac{\beta^s}{s}\,\|u\|^s=C\rho^{G^--1}-\lambda\,\frac{\beta^s}{s}\,\rho^s.$$
Taking higher and higher values of $\rho$, we deduce our statement for all $\lambda>0$.
\end{proof}

By Lemma \ref{lema} we can deduce that there exists  big enough $\rho>0$ such that $$\inf\{\ee(u):\ u\in\WW,\, \|u\|\leq\rho\}=:m<0$$and $$\sup\{\ee(u):\ u\in\WW,\,\|u\|=\rho\}>0.$$

Let $M$ be the complete metric space defined by $M:=\{ u\in\WW:\ \|u\|\leq\rho\}$. Applying Ekeland's variational principle to $\ee$ restricted to $M$ we find a sequence $(u_n)\subset\WW$ of ``almost critical points" of $\ee$, that is,
\bb\label{munte1}\ee(u_n)\ri m<0\quad\mbox{and}\quad\ee'(u_n)\ri 0\quad\mbox{as}\ n\ri\infty.\bbb

Since $(u_n)$ is bounded it follows that, up to a subsequence, $u_n\rightharpoonup u$ in $\WW$. Next, by Lemma~\ref{comgr}, we can assume that $u_n\ri u$ in $L^s(\Omega)$. Combining this information with \eqref{munte1} and the fact that $\ee'$ is a mapping of type $(S_+)$, we deduce that $u_n\ri u$ in $\WW$, hence $u$ is a nontrivial solution of problem~\eqref{imain}.
This concludes the proof of Theorem~\ref{thm}.

\section*{Concluding remarks, perspectives, and open problems} (i) Patterns  and  waves  are  all  around  us.  They occur in   many   different   systems  and  on  vastly  different    scales    in    both time  and  space,  and  their
dynamic behavior is similar  across  these  systems. Mathematical   techniques can  help  identify  the  origins  and  common  properties  of  patterns and  waves  across  various  applications.  Understanding  the  ways  in  which  such structures  are  created  can  help  experimentalists  identify  the  mechanisms  that
generate  them  in  specific  systems.
Despite  many  advances,  understanding patterns  and  waves associated to transonic flows  still  poses  significant
mathematical challenges. To illustrate the difficulties, we have considered in this paper  a  case  in  which  one
has  developed  a  partial  differential  equation that possesses both elliptic and hyperbolic regions.
The flow is supersonic in the elliptic region, while a shock wave is created at the boundary between the elliptic and hyperbolic regions.
The model studied in this paper is described by a Baouendi--Grushin operator, which
can be seen as the Tricomi operator in the framework of the transonic flow restricted to subsonic regions.
  We were interested  in
assessing  the existence of stationary waves under arbitrary  perturbations  of  the reaction, which corresponds
to the study of a suitable nonlinear eigenvalue problem.

\smallskip\noindent
(ii) The mathematical analysis carried out in this paper considers the unbalanced energy defined in \eqref{uu1} with the associated differential operator defined in \eqref{uu2}. It appears to be worth to further investigate patterns described by the variational integral
\bb\label{uu3}\int_{\Omega}\left(\left|{\rm grad}_{x}u\right|^{G(x,y)}+
\left|x\right|^{\gamma}\left|{\rm grad}_{y}u\right|^{G(x,y)}\right)dz\bbb
with corresponding anisotropic Baouendi-Grushin operator
$${\rm div}_{x}\left(G(x,y)\,|{\rm grad}_{x} |^{G(x,y)-2}{\rm grad}_{x} \right)+{\rm div}_{y}\left(G(x,y)\,|x|^{\gamma}\left|{\rm grad}_{y} \right|^{G(x,y)-2}{\rm grad}_{y}\right).$$

\smallskip\noindent
(iii) We remark that since the  energy functionals introduced in relations \eqref{uu1} and \eqref{uu3}
have a degenerate action on the set where the gradient vanishes, it is a natural question to study what
happens if the integrand is modified in such a way that, if $|{\rm grad} u|$ is also small, there exists
an imbalance between the two terms of every integrand.

\smallskip\noindent
(iv) Lemma \ref{comgr} played a key role in the proof of the existence of an interval of eigenvalues for problem \eqref{imain}. This compactness property is established in a {\it subcritical} setting, which corresponds to
the hypothesis $s<G^-$, where $s$ describes the growth of the right-hand side of problem \eqref{imain}. In fact, Theorem \ref{thm} remains true if $s$ is replaced with a variable coefficient $s(x)$, provided that $s^+<G^-$.
We do not have any knowledge about the behaviour in the {\it almost critical} case that arises in the following situation: there exists $x_0\in\Omega$ such that $s(x_0)=G^-$ and $s(x)<G^-$ for all $x\in\Omega\setminus\{x_0\}$.

\smallskip\noindent
(v) The weighted inequality established in Theorem \ref{cafg} is stated under the hypothesis that the variable coefficient is of class $C^1$. We consider that a valuable research direction (with relevant applications to nonsmooth mechanics) corresponds to potentials leading to a lack of regularity. For instance, if $G$ is locally Lipschitz, one can use the notion of {\it Clarke generalized gradient}. We do not know of any version of Theorem~\ref{cafg}
for potentials $G$ having loss of regularity.

\smallskip\noindent
(vi) In a forthcoming paper, we will study new classes of nonlinear boundary value problems involving the magnetic Baouendi-Grushin operator, see \cite{laptev, laptev1}. This operator is
$$G_{\mathcal A}:=-({\rm grad}_G+i\beta{\mathcal A}_0)^2\quad\mbox{for}\ -\half\leq\beta\leq\half,$$
where
$${\mathcal A}_0=({\mathcal A}_1,{\mathcal A}_2,{\mathcal A}_3,{\mathcal A}_4)=\left(-\frac{\partial_yd}{d}, \frac{\partial_xd}{d},-2y\frac{\partial_td}{d},2x\frac{\partial_td}{d} \right),$$
$${\rm grad}_G=(\partial_x,\partial_y,2x\partial_t,2y\partial_t),$$ with $z=(x,y)$, $|z|=\sqrt{x^2+y^2}$, and $d(z,t)=(|z|^4+t^2)^{1/4}$ is the Kaplan distance.

\smallskip\noindent
(vii) We believe that the approaches and techniques developed for studying problem \eqref{imain} can be useful for the qualitative analysis of further classes of nonlinear problems described by mixed type operators, either stationary or evolutionary. These degenerate or singular problems include the von
Neumann problem (which describes the shock reflection-diffraction by two-dimensional wedges with concave
corner), the Lighthill problem (which describes the shock diffraction by two-dimensional wedges with convex corner),
and the Prandtl-Meyer problem (in the framework of supersonic flows impinging onto solid wedges). These models
 describe very relevant phenomena that arise in fluid mechanics. At the same time, they are also fundamental mathematical models in the
 theory of multidimensional conservation laws. These reflection/diffraction
configurations are the core configurations in the structure of global entropy solutions of the two-dimensional
Riemann problem for hyperbolic conservation laws, whereas the Riemann solutions
are the building blocks and local structure of general solutions, and determine global attractors
and asymptotic states of entropy solutions, as time tends to infinity, for multidimensional
hyperbolic systems of conservation laws. In this sense, we have
to understand the reflection/diffraction phenomena in order to fully comprehend global
entropy solutions to multidimensional hyperbolic systems of conservation laws.

\medskip
{\bf Acknowledgments.}
The research of V.D.~R\u adulescu and D.D.~Repov\v{s} was supported by the Slovenian Research Agency grants
P1-0292, J1-8131, J1-7025, N1-0064, and N1-0083. V.D.~R\u adulescu also acknowledges the support through the Project MTM2017-85449-P of the DGISPI (Spain).

\end{document}